\documentclass{amsart}

\usepackage[left=3cm,top=3cm,right=3cm,bottom=3cm]{geometry}
\usepackage{amssymb, amsmath, amsthm}
\usepackage{verbatim}
\usepackage{comment}
\usepackage{graphicx}
\usepackage{enumerate}
\usepackage{color}

\DeclareGraphicsExtensions{.pdf,.jpg,.png}

\theoremstyle{plain}
\newtheorem{theorem}{Theorem}

\newtheorem{claim}[theorem]{Claim}

\theoremstyle{definition}
\newtheorem{definition}[theorem]{Definition}

\newcommand{\Z}{\mathbb{Z}}

\makeatletter

\begin{document}

\title{Integral and rational graphs in the plane}

\author{J\'{o}zsef Solymosi}
\address{\noindent Department of Mathematics, 
University of British Columbia, 1984 Mathematics Road, 
Vancouver, BC, V6T 1Z2, Canada and Obuda University, H-1034 Budapest, Hungary }
\email{solymosi@math.ubc.ca}

\maketitle

\begin{abstract}
We describe constructions of infinite graphs which are not representable as integral graphs in the plane, addressing
a question of Erd\H{o}s. We also mention some related problems.
\end{abstract}

\section{Introduction}

Anning and Erd\H{o}s proved in \cite{AE} that if an infinite point set in the plane has integer pairwise distances only, then the points are collinear. 
Shortly after their result, Erd\H{o}s found a simple and elegant proof \cite{Er}. This proof - which we will sketch in the first section - was the base of further results about integral point sets\footnote{a set of points is an {\em integral point set} if all pairwise distances are integers}.

We are going to analyze integral (rational) distance graphs. These are geometric graphs where the vertices are points in the plane, no three on a line and two are connected by an edge if and only if their distance is an integer (rational) number. Erd\H{o}s asked if there is a way to characterize integral graphs. In the same paper, \cite{ErKomal}, he noted that his proof implies that such graphs contain no $K(3,\aleph_0)$, i.e. three vertices $X_1, X_2, X_3$ and infinitely many vertices $Y_1, Y_2, \ldots$ such that $X_i$ and $Y_j$ are connected whenever $1\leq i\leq 3$ and $1\leq j.$ He asked if graphs without $K(3,\aleph_0)$ are integral graphs. Maehara, Ote and Tokushige proved that every finite graph is an integer graph \cite{MOT}. This leaves the question, what infinite graphs can be realized as integral graphs? We will construct families of graphs which can not be subgraphs of integral graphs. In light of these examples, the complete characterization of integral graphs seems difficult. 

Even less is known about rational graphs. One can choose a dense subset of the unit circle such that all pairwise distances are rational, so $K_{\aleph_0}$ is a rational graph. Maehara had a nice construction showing that the infinite complete bipartite graph, $K(\aleph_0,\aleph_0),$ is rational \cite{Ma}. It is not known if every finite graph is a rational graph or not. We will show that a major conjecture in arithmetic geometry, the Bombieri-Lang conjecture \cite{HS}, would imply that there are simple graphs which are not rational.

\section{Non-integral graphs}

We are going to show three different constructions for non-integral graphs. The first example is based on the observation that in any integral graph, the common neighbours of any two vertices span a subgraph with a finite chromatic number. The second graph has a very simple structure; every vertex has a finite degree, and the chromatic number of the graph is six. The third one is a combination of the previous two. While it has an unbounded chromatic number, its vertices have finite degrees. All three constructions are built on the ``hyperbola method'' originated from Erd\H{o}s. Before the constructions, let's review Erd\H{o}s' proof for the Anning-Erd\H{o}s theorem from \cite{Er}.

\begin{theorem}[Anning and Erd\H{o}s] 
If $P$ is an infinite point set in the plane, such that the distances are integers, $\overline{AB} \in \Z$ for any $A,B\in P$, then $P$ is a subset of a line.     
\end{theorem}
\begin{proof}
Let $A,B,C\in P$ be three non-collinear points. If $X\in P$ then $|\overline{AX}-\overline{BX}|$ is an integer between $0$ and $\overline{AB}$.
Then $X$ lies on one of the $\overline{AB}+1$ hyperbolas, defined by $|\overline{AX}-\overline{BX}|=D$ where $D\in \{0,\ldots,\overline{AB}\}$.
Similarly, $X$ lies on one of the $\overline{AC}+1$ hyperbolas, defined by $|\overline{AX}-\overline{CX}|=H$ where $H\in \{0,\ldots,\overline{AC}\}$.
The two sets of hyperbolas intersect in at most $F(A,B,C)=4(\overline{AB}+1)(\overline{AC}+1)$ points, so $|P|\leq F(A,B,C),$ a finite number.
\end{proof}

Note that we didn't use the assumption that $\overline{AB}$ and $\overline{AC}$ are integers. So, as Erd\H{o}s observed in \cite{ErKomal}, integral graphs do not contain subgraphs isomorphic to $K(3,\aleph_0).$ This was also noted by Maehara in \cite{Ma}, where he made further observations about integral and rational graphs.

\medskip
The following examples show that there are graphs free of copies of $K(3,\aleph_0)$ but not integral graphs.

\subsection{The first family of non-integral graphs}
In the first construction, we are using geometric properties of hyperbolas. The canonical form of a hyperbola is given by the equation

\[
\frac{x^2}{a^2}-\frac{y^2}{b^2}=1.
\]

The points of a hyperbola can be partitioned into four arcs based on the signs of their $x,y$ coordinates. Two points of the hyperbola, $P=(x_1,y_1)$ and $Q=(x_2,y_2)$ are in the same partition class if $sign(x_1)=sign(x_2)$ and $sign(y_1)=sign(y_2),$ i.e. the two points are in the same quadrant of the real plane.

\begin{claim}\label{cross}
If two points, $P$ and $Q$, of a hyperbola are on the same arc, then the line through them intersects the interval $[-a,a].$ (see Fig. \ref{cross_pic}).
\end{claim}

\begin{proof}
    The hyperbola divides the plane into three connected regions—two, where the two focus points lie and a ``middle'' region in between them. The $PQ$ line has both intersection points with the hyperbola in one quarter, so it will intersect the $F_1F_2$ interval in the middle region.

\begin{figure}[h]

\centering
\includegraphics[width=0.6\textwidth]{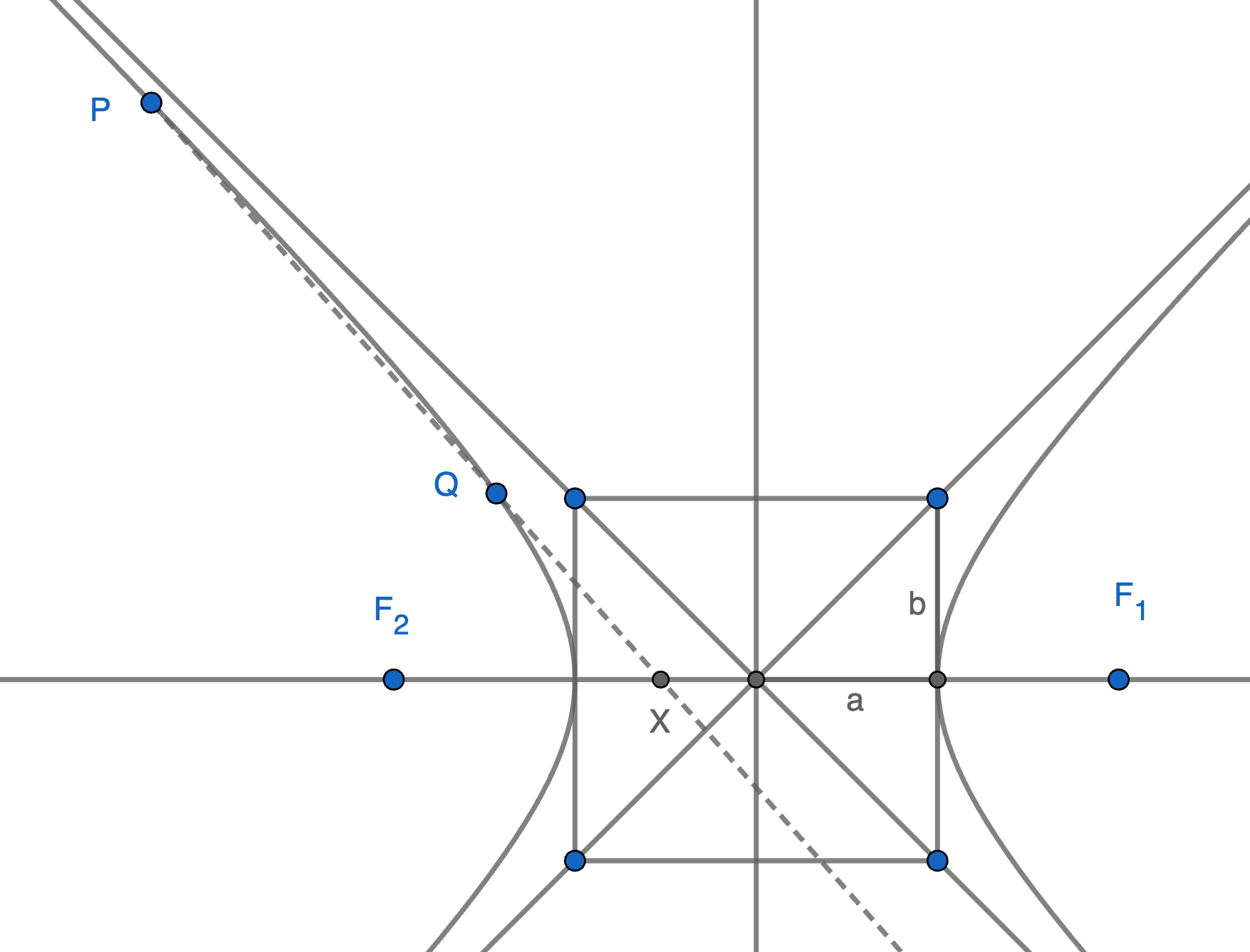}
\caption{The line through $P$ and $Q$ intersects the $x$-axis between $F_1$ and $F_2$}
\label{cross_pic}
\end{figure}

\end{proof}

\medskip

\begin{claim}\label{neighbours}
    In an integral graph, for any two distinct vertices $v_i,v_j$, the graph spaned on the vertices of their common neighbours, $N(v_i)\cap N(v_j)$, has a finite chromatic number.
\end{claim}

\begin{proof}
    Every integral triangle (a triangle with integer side lengths), $a\leq b\leq c,$ satisfies the inequality $a+b\geq c+1.$ A simple calculation shows that the triangle's height is at least $\sqrt{a-1/4}.$ 
    
    It follows that if $P$ and $Q$ are on the same arc and the distance $\overline{F_1F_2}\leq\sqrt{a-1/4}$ and 
    
    $$\min\left\{\overline{PQ},\overline{F_1Q},\overline{F_1P}\right\}\geq a,$$ then $PQF_1$ is not an integral triangle. Let $G$ be an integral graph with a realization in the plane. If $v_1,v_2\in V(G),$ let's denote the corresponding points in the plane by $F_1$ and $F_2.$ 
    We can assume that $F_1=(-c,0)$ and $F_2=(c,0)$ with some $c\in \mathbb{R}.$
    We will colour the vertices of 
    $N(v_i)\cap N(v_j)$ based on their point's positions relative to $F_1,F_2.$ We have three rounds of colouring. First, colour the points (vertices) according to their hyperbola arcs. We used at most $4\overline{F_1F_2}$ colours. Then, in every arc, we give different colours of the point pairs, which are at integer distances and closer than $$d=\overline{F_1F_2}^2+1.$$
    There are various ways to perform such colourings. Here, we show it can be done using $2d$ colours at most. This colouring is even independent of the graph. Let's define an infinite graph on a hyperbola arc. The vertices are the points of the hyperbola, and two points are connected if their distance is a positive integer, which is smaller than $d.$
    Every vertex has a degree at most $2d$ so we need at most $2d$ colours. (Here, we were using the De Bruin-Erd\H{o}s theorem about colouring infinite graphs \cite{EDB}). One could show that fewer colours are enough for a good colouring, but here, we only care about the finiteness of this number. In the third colouring, we give different colours to all the points which are at an integer distance from $F_1$ and $F_2$, and one of these distances is less than $d.$
    
    Combining the three colourings, we have a finite number of colours assigned to the vertices. It is a good colouring of the subgraph on $N(v_i)\cap N(v_j).$ Indeed, let's suppose there is an edge with the same colours in its vertices, $v_i,v_j.$ That would mean that the relevant points, $P$ and $Q,$ are on the same arc, forming an integral triangle with $F_1$ and $F_2,$ and side lengths at least $d.$ But it is impossible since any such triangle has a height larger than $\overline{F_1F_2}.$
    \end{proof}

Based on Claim \ref{neighbours}, we can construct a graph which can't be realized as an integer graph. Take an arbitrary graph, $G$, with an infinitely large chromatic number. Connect all vertices of $G$ with two additional vertices $v$ and $w.$ This graph has at least two vertices, $v$ and $w$ with infinite degree. In the next construction, every vertex has a finite degree.

\subsection{The second family of non-integral graphs}
The second graph has a ``spine'', a sequence of triangles $T_1, T_2, ..., T_i,\ldots $ where consecutive triangles are connected by all nine possible edges ($T_i$ and $T_{i+1}$ span a $K_6$).
Every triangle, $T_i$, is connected to $n_i$ vertices. These vertices, denoted by $P_i$, are not connected to anything else. They all have degree three. (Fig. \ref{treelike})

\medskip

\begin{figure}[h]

\centering
\includegraphics[width=0.6\textwidth]{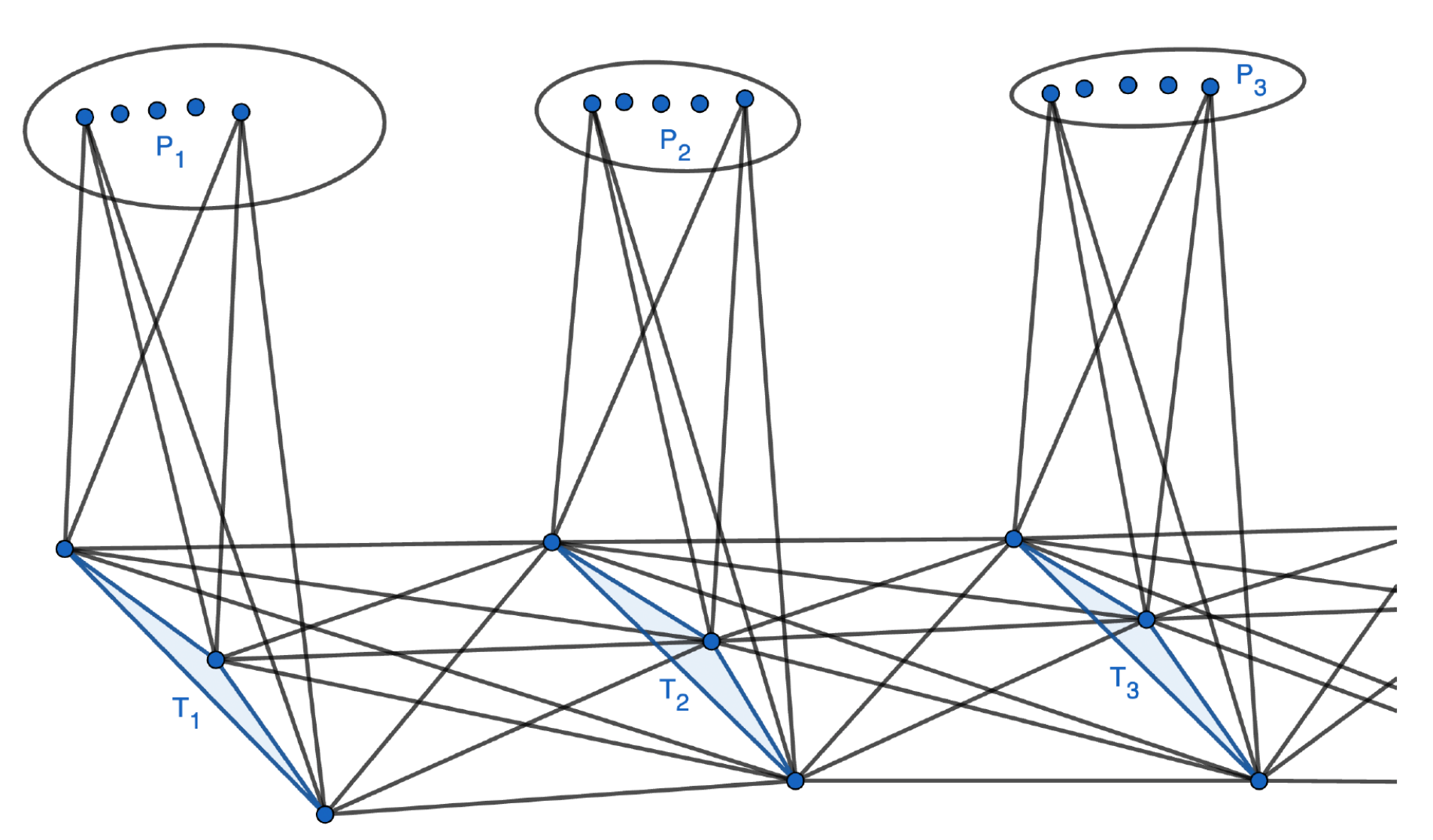 }
\caption{With a fast-growing $|P_i|$ sequence this graph is not integral}
\label{treelike}
\end{figure}

\medskip

By choosing the properly increasing $n_i$ values, we can guarantee that this graph can't be represented as an integral graph.

If, in a potential realization of the integer graph, the perimeter of $T_i$ is at most $p_i$, then it gives a bound on the max possible perimeter of $T_{i+1}$. We have $p_{i+1}<F(p_i)$ for some function $F$. Indeed, there are less than $p_i^2$ integral triangles with perimeter $p_i,$ and for each, the intersection points of the two sets of hyperbolas from Erd\H{o}s' proof are bounded. They all have a finite diameter, so any triangle spanned by these points has a bound on its perimeter in terms of $p_i.$ 

We also have a bound on $n_i$ for a given $p_i.$ It is given by the number of intersections of the hyperbolas given by two integer sides. It is less than $4p_i^2.$
Let $L(j,p)$ denote the bound on $n_j$ if the first triangle, $T_1$, had perimeter $p$. This function is well-defined for every $j$ and $p$ positive integer. Now, in our graph, we define $n_i:=L(i,i)+1$.
In any realization of the graph, $T_1$ has a perimeter, say $p$, so $n_p$ is larger than the possible bound given by $L(i,p)$.

This non-integral graph has simple graph properties: every degree is finite, and its chromatic number is six.

\medskip

\subsection{The third family of non-integral graphs}
This construction is an example of how to combine the previous two non-integer graphs to get a new type of graph. 

As in the second non-integer graph, we have an infinite set of triangles, $T_1, T_2, ..., T_i,\ldots $, where the consecutive triangles form a $K_6.$ 

From the first construction, we know that there is a function $H(d),$ such that if $\overline{PQ}\leq d,$ then in the graph spanned common neighbours of the vertices represented by $P$ and $Q$ span an integral graph $G'$ with $\chi(G')\leq H(d).$

\begin{figure}[h]

\centering
\includegraphics[width=0.6\textwidth]{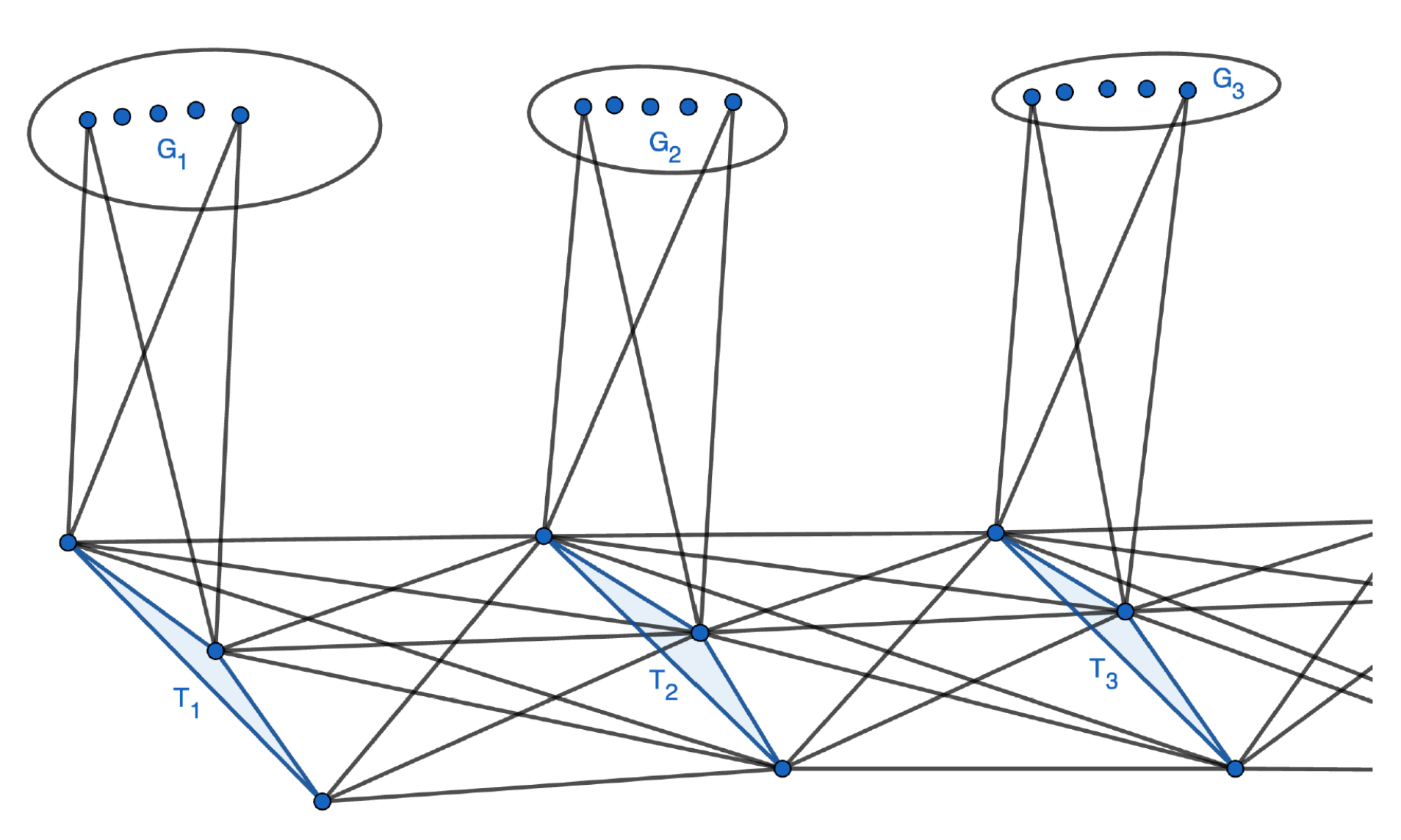 }
\caption{With a fast-growing $\chi(G_i)$ sequence this graph is not integral}
\end{figure}

\medskip

Similar to the earlier case, there is a function, $F^*$, such that if the first triangle, $T_1$, has a perimeter $p$, then the perimeter of $T_i$ is at most $F^*(p,i)$. Now, we choose a sequence of graphs with fast-growing chromatic numbers. Let's choose $G_i$ such that $\chi(G_i)\geq H(F^*(i,i)).$ For any starting $T_1$ tringle, there will be an index where $\chi(G_i)$ is larger than it would be allowed in an integral representation.

\section{Rational graphs}
A graph is a rational graph if it can be represented by points in the plane as vertices, and two points are connected by an edge if and only if their distance is a rational number. As for the integral graphs, no three points are allowed to be collinear.
While every finite graph is an integral graph, rational graphs are harder to construct. In \cite{MOT}, the construction of integral graphs consisted of points on a circle such that all pairwise distances were rational. The hyperbole method from the previous section won't work for rational graphs.
There is a very limited set of graphs known to be rational. 
Maehara proved that $K(\aleph_0,\aleph_0)$ is a rational graph.
A result of Berry \cite{Be} states that if a triangle has side lengths $a, b$ and $c$ such that $a^2,b^2$ and $c$ are all rational numbers, then the set of points which are at rational distances from all three vertices of the triangle are everywhere dense in the plane. Using this, one can construct families of rational graphs with some strong degree restrictions. See, e.g. in \cite{GGM} and \cite{VIB}. While only very special graphs are known to be rational graphs, no graph is known which doesn't have a rational representation.

A conjecture of Erd\H{o}s and Ulam states that there is no everywhere dense set of points in the plane such that all pairwise distances are rationals. Assuming the Bombieri-Lang conjecture, it was proved in \cite{Ta} and independently in \cite{Sh} that there are no such sets. It was also proved using the ABC conjecture in \cite{Pa}. Ascher, Braune, and Turchet proved an even stronger statement \cite{ABT}. They used the following definition for points in general position.

\begin{definition}
    We say that $n\geq 4$ points in the plane are in general position if no $n-4$ are on a line and no $n-3$ are on a circle.
\end{definition}

It was proved in \cite{SdZ} that if an algebraic curve contains infinitely many points of a rational point set (i.e. all pairwise distances are rationals), then all but four points are on a line, or all but three are on a circle. Moreover, such rational sets exist. In \cite{ABT}, they proved that assuming the Bombieri-Lang conjecture, there is a bound $N,$ such that if a rational set contains at least $N$ points, then all but at most four on a line, or all but at most three are on a circle. Based on this result, we can construct a simple non-rational graph (assuming the Bombieri-Lang conjecture).

In the definition of integral and rational graphs, collinear triples were forbidden in the realization of such graphs. So, a complete rational graph on $N+4$ vertices could only be represented as at least $N+1$ points lying on a circle, $C$. Let $G$ be a graph, which consists of a $K_N$ and two disjoint $K_4$-s, both connected to $K_n$ with all $4n$ edges, i.e. two $K_{N+4}$ with $N$ common vertices. By the assumption, at least one vertex from each $K_4$-s is on $C.$ Let's denote the two points by $P$ and $Q,$ and two other points on the circle from the $K_N$ part, $R$ and $S.$ All of the distances, $\overline{RS}, \overline{PR}$, $\overline{QR},\overline{SQ},\overline{PQ}$ are rational. But since the four points are on the same circle, $\overline{PQ}$ should also be rational by Ptolemy's theorem. (The product of the diagonals equals the sum of the products of the opposite sides)

That means $G$ is not a rational graph if we assume the Bombieri-Lang conjecture. 

\section{Open problems}
Many questions remain open. 
\begin{enumerate}
    \item Find a graph, finite or infinite, which is not a rational graph without assuming conjectures.
    \item Every finite graph is an integral graph, but what is the minimal diameter of a realization of a given graph? It was recently proved in \cite{GIP} that $K_n$ always has a large diameter.
    \item A nice conjecture was stated by Kemnitz and Harboth \cite{KH}, that there exists a plane integral drawing for every planar graph. See the proof for some planar graphs in \cite{GGM}
    \item What are the graphs which can be represented by odd distances? For references, check \cite{Da}.
    \item Representing integral graphs becomes a much harder problem if we forbid co-circular four-tuples in addition to collinear triples. It is open if $K_8$ has such representation. For such restricted representation of $K_7$, we refer to \cite{KK} and \cite{MTX}.
\end{enumerate}

Some of these problems and more can be found in the problem book by Brass, Moser and Pach \cite{BMP}.

\section{Acknowledgements}
This research was partially supported by the Simons Foundation and the Mathematisches Forschungsinstitut Oberwolfach and by an NSERC grant.

\end{document}